\documentclass[11pt,a4paper]{article}

\usepackage{authblk}
\usepackage[multiple]{footmisc}

\usepackage{array}
\newcolumntype{P}[1]{>{\centering\arraybackslash}p{#1}}
\newcolumntype{M}[1]{>{\centering\arraybackslash}m{#1}}

\usepackage[toc,page]{appendix}

\usepackage{multirow}
\usepackage{graphicx}

\usepackage{amsfonts,amsmath,amssymb,amsthm,mathtools}
\usepackage[english]{babel}
\usepackage[utf8]{inputenc}

\usepackage[LGR,T1]{fontenc}

\DeclareMathOperator*{\argmax}{arg\,max}

\usepackage{afterpage}
\usepackage{wrapfig}
\usepackage{dsfont}
\usepackage{imakeidx}
\usepackage[ruled]{algorithm2e}

\usepackage[framed, numbered]{matlab-prettifier}

\usepackage{caption,subcaption}

\usepackage{tabularx}
\usepackage{caption}
\usepackage{float}

\usepackage[shortlabels]{enumitem}

\usepackage{lettrine}

\newtheorem{asm}{Assumption}
\newtheorem{rmk}{Remark}
\newtheorem{thm}{Theorem}
\newtheorem{lemma}{Lemma}
\newtheorem{clm}{Claim}
\newtheorem{propos}{Proposition}
\newtheorem{clr}{Corollary}

\numberwithin{equation}{section}
\numberwithin{propos}{section}
\numberwithin{rmk}{section}
\numberwithin{lemma}{section}
\numberwithin{clr}{section}
\numberwithin{Def}{section}

\newcommand{\PP}{\mathbb{P}}
\newcommand{\EE}{\mathbb{E}}
\newcommand{\NN}{\mathbb{N}}

\newcommand{\eps}{\varepsilon}

\newcommand{\la}{\lambda}

\makeatletter
\let\@fnsymbol\@arabic
\makeatother

\begin{document}

\title{Optimal control, viscosity approximation and Arrhenius Law for the shallow lake problem}


\author{Angeliki Koutsimpela\footnote{Department of Mathematics, University of Augsburg, 86159 Augsburg, Germany} \, and Michail Loulakis\footnote{School of Applied Mathematical and Physical Sciences, National Technical University of Athens, 15780 Athens, Greece}$^{\; \; , }$\footnote{Institute of Applied and Computational Mathematics, Foundation for Research and Technology Hellas,
			 70013 Heraklion Crete, Greece}}



\maketitle

\begin{abstract}
We prove existence of optimal control for the deterministic and stochastic shallow lake problem without any restrictions on the parameter space and we establish a generalization of the Arrhenius Law in the case of noise-dependent potentials, which naturally arise in control theory problems. We also prove a result about convergence of the derivatives in the viscosity approximation of the value function and use this result to derive the Arrhenius Law for the shallow lake problem. 
\end{abstract}

\section{Introduction}

The shallow lake problem is a well-known problem of the environmental  economy with a great mathematical interest. Pollution of shallow lakes is caused by human activity, e.g. the use of fertilizers and the increased inflow of waste water from industries and human settlements, and is usually quantified by the concentration of phosphorus into the lake.   The amount of phosphorus in algae is usually modeled by the non-linear stochastic differential equation:
\begin{equation} \label{sldyn}
	\begin{cases} 
		dx(t)= \left( u(t)-b x(t)+r(x(t))\right) dt+\sigma x(t)dW_t, \\
		x(0)= x\ge 0.  
	\end{cases}
\end{equation}
The first term, $u:\; [0,\infty)\rightarrow (0,\infty)$, in the drift part of the dynamics, represents the exterior load of phosphorus as a result of human activities. The second term is the rate of loss $bx(t)$, which is due to sedimentation, outflow and sequestration in other biomass. The third term, $r(x(t))$, is the rate of recycling of phosphorus on the bed of the lake. The function $r$ according to limnologists is best described by a sigmoid function (see \cite{carp99}) and the typical choice in the literature is the function $x\mapsto x^2/(x^2+1)$. An uncertainty in the rate of loss is inserted in the model through a linear multiplicative Gaussian white noise with intensity $\sigma \neq 0$.

The economics of the lake arise from the conflicting services it offers to the community. On the one hand, a clear lake is a resource for ecological services. On the other hand, the lake serves as a waste sink for the agricultural and industrial activities. When the users of the lake cooperate, the loading strategy can be used as a control to maximize the benefit from the lake. Assuming an infinite horizon, this benefit is typically defined as
\begin{equation}\label{Jxu}
	J(x;u)=\EE_x \left[\int_0^\infty e^{-\rho t}\big(\ln u(t)-cx^2(t)\big)\, dt \right],
\end{equation}
where $\rho >0$ is the discount rate and $x(\cdot)$ is the solution to (\ref{sldyn}), for a given exterior load (control)  $u$, and  a given initial state $x\geq 0$. The total benefit of the lake increases with the increase of loading of phosphorus as $\ln u(t)$, but at the same time decreases with the amount of phosphorus inside the lake as $-cx^2(t)$, due to the implied deterioration of its ecological services. The positive parameter $c$ reflects the relative weight of this component.

For the optimal management of the lake, when the initial concentration of phosphorus is $x$, we need to maximize the total benefit over all admissible controls $u\in\mathfrak{U}_x$. In this way, the value function of the problem is defined as 
\begin{equation}\label{ihvf}
	V(x) = \sup_{u\in \mathcal{U}_x}  J(x;u),
\end{equation}
\vskip.05in
Therefore, the shallow lake problem becomes a problem of control theory or a differential game in the case where we have competitive users of the lake \cite{Brock, carp99,  Xep}.

The shallow lake problem has been extensively studied in the literature, especially its deterministic version, where $\sigma=0$. The case where  the optimally controlled lake has two equilibria and a Skiba point \cite{skiba, W03} is of particular interest. 
The leftmost (oligotrophic) equilibrium corresponds to a lake with low concentration of phosphorus, while the rightmost one (eutrophic) corresponds to a lake with high concentration of phosphorus. At the Skiba point, there are two different optimal strategies, each one driving the system to a different equilibrium and the value function is not differentiable thereat. Extensive exploration of the parameter space and the qualitative differences of the Pontryagin system of the shallow lake (bifurcation analysis) has been conducted \cite{WBif, W03}. Properties of the value function have been proved in  \cite{KosZoh}.

A basic question that, to the best of our knowledge, has not been completely answered so far is that  of the existence of optimal control. The existence of optimal control is usually taken as a hypothesis and the optimal dynamics of the lake is studied mostly through the necessary conditions, which are determined by the Pontryagin Maximum Principle, and the equilibrium points of the corresponding dynamical system  \cite{W03, Xep}.  A rigorous answer to this question was given by Bartaloni in \cite{Bar20, Bar21}, albeit under restrictions that do not fully cover the range of the parameters for which  Skiba points are present. 

In addition, there has been increasing interest recently in the stochastic  ($\sigma\neq 0$) version of the problem.  Specifically, deterministic systems with two equilibrium points and one Skiba point have a fundamentally different behaviour  from their stochastic counterparts. In particular, in the presence of noise, random fluctuations lead the system from the one equilibrium point to the other one (metastability). In the case of a shallow lake, variations in the rate of loss drive the lake from the oligotrophic to the eutrophic state and vice versa, a phenomenon which is naturally observed.  The interest in the study of metastable systems firstly arose from phenomena in the field of chemistry. Arrhenius \cite{Arrh} in 1889 physically justified an expression for the mean transition time of the system to go from the one local minimum to the other. Later, H. Eyring and H. A. Kramers \cite{Eyr}, \cite{Kram} with the well-known Eyring-Kramers law specified the prefactor term in Arrhenius' expression. In the sequel, M.I.Freindlin and A.D.Wentzell \cite{FRWE} introduced the theory of large deviations to explain and understand the metastable behaviour of various dynamical systems. Even though,  metastable systems have been extensively studied ever since (see e.g. \cite{AH}, \cite{Be}) , the majority of results concern dynamical systems for which the drift function is not a function of the noise intensity.  However, in the context of (stochastic) control theory, one naturally expects that in metastable systems, like the shallow lake system in the case of Skiba points, the drift term of the optimally controlled system will on the value function, which in turn depends on the noise via the HJB correction. The phenomenon of metastability in the shallow lake problem is studied numerically in \cite{GKW}, where the value function of the shallow lake problem is approximated for small $\sigma$, based on heuristic methods of perturbation analysis. 

Furthermore, thorough examination of the stochastic version of the shallow lake problem is conducted by Kossioris, Loulakis and Souganidis \cite{KLS} who  analytically derive properties of the value function and characterise it as the unique (in a suitable class) state-constraint viscosity solution of the Hamilton-Jacobi-Bellman (HJB) equation
\begin{equation}\label{OHJBsup}
\rho V- H(x,V_x)-\frac{1}{2}\sigma^2 x^2 V_{xx}=0
\end{equation}
where $H(x,p)=\sup \limits_{u\in (0,\infty)}\tilde{H}(x,u,p):=\sup \limits_{u\in (0,\infty)} \{ \left( u-bx+r(x) \right)p+\ln u-cx^2 \}$.\\
\noindent 
Their results concern the usual choice of the recycling function $r(x)=\frac{x^2}{x^2+1}$ and the choice $c=1$ for the weight component appearing inside the total benefit. In the following, their results are extended in order to involve more general sigmoid functions $r$ and values of the parameter $c$ in \cite{KL}. Proving that the derivatives of the value function $V$ are actually negative and bounded away  from zero
 reduces eq. \eqref{OHJBsup} to the following form:
\begin{equation}\label{OHJB}
\rho V-\left(r(x)-bx\right)V_x+ \ln(-V_x)+cx^{2}+1 -\frac{1}
{2}\sigma^{2}x^{2}V_{xx},
\end{equation}

The connection of control theory problems with  HJB equations has been extensively studied, see e.g \cite{Bardi,FS,Flem,Lions2,Louis}. In particular, when the value function does not a priori possess  the regularity of a classical solution, the theory of viscosity solutions, as it was developed by Crandall and Lions \cite{Crandall}, offers a flexible framework which can also cover problems with state-constraints.

However, the shallow lake problem has some nonstandard features and, hence, it requires some special
analysis. First of all, the control space is open and unbounded, so  the usual assumptions  made to prove existence in control problems with infinite horizon  (see e.g.  \cite{Besov, Dmit, Balder}) are not satisfied here. In addition, a priori knowledge of the properties of the solution is necessary to guaranty that the Hamiltonian, $H$, is well-defined, due to the logarithmic term in \eqref{OHJB}.  Finally, in the case of the stochastic shallow lake problem, ellipticity of \eqref{OHJB} degenerates at the boundary, $x=0$.

In this paper, we prove existence of optimal control for both the stochastic and the deterministic shallow lake problems. In the presence of noise, the proof follows the general lines of a verification principle (see e.g. \cite{FS}) with appropriate modifications to address the loss of ellipticity at the boundary and a possible blow-up of the benefit for small controls. This approach is not always feasible in the deterministic problem, since the value function may fail to be everywhere differentiable. In \cite{Bar20} and \cite{Bar21} the existence of optimal control is established by proving uniform localization lemmas followed by diagonal arguments. This approach is successfully carried out under the assumption that either the parameter $b$ or the discount parameter $\rho$ are greater than $3\sqrt{3}/8$. Our approach here is entirely different and does not require any restrictions on the parameter space. More specifically, we prove that both the value function and the total benefit achieved when the system is driven by the candidate optimal control suggested by Pontryagin Maximum Principle are viscosity solutions to the same well-posed problem.  In this way, the above-mentioned candidate optimal control is proved to be indeed optimal.
%
Finally, in the last part of the article, we prove a generalization of the Arrhenius' law in the framework of control theory problems with Skiba points, where we have a noise-dependent potential, and consider the shallow lake problem as an application of this result.


\section{The general setting and the main results}

\subsection{Existence of optimal control}

We assume that there exists a filtered probability space $( \Omega, \mathcal{F},\{ \mathcal{F}_t\}_{t\ge 0},\PP )$
satisfying the usual conditions, and a Brownian motion $W(\cdot)$ defined on that space. 
An admissible control   $u(\cdot)\in\mathfrak{U}_x$ is an $\mathcal{F}_t$-adapted, $\PP$-a.s. locally integrable process with values in $U=(0,\infty)$, satisfying 
\begin{equation}\label{ac_constraint}
	\mathbb{E} \left[ \int_{0}^{\infty}e^{-\rho t}\ln u(t)dt \right] < \infty,
\end{equation}
such that  the problem (\ref{sldyn}) has a unique strong solution $x(\cdot)$. Furthermore, throughout the paper, the recycling rate function $r$ is a sigmoid function satisfying Assumption \ref{r}.
\vskip.05in

\begin{asm}\label{r}
The rate of recycling $r$ satisfies the following:
\begin{enumerate}
	\item $r\in C^1([0,\infty))$ and non-decreasing
	\item  $r(0)=0$ and $r(x)\leq (b+\rho)x$ close to $0.$
	\item $a:=\lim \limits_{x\rightarrow \infty}r(x)<\infty$
	\item The limit $\lim \limits_{x\rightarrow \infty}(a-r(x))x \in \mathbb{R}=:C$ exists and is a finite, necessarily nonnegative, real number
    \item $\lim \limits_{x\rightarrow \infty}r'(x)=0$
\end{enumerate}
\end{asm}

\noindent 
Next, we present the main results regarding the existence of optimal control. 

\noindent

\begin{thm}\label{constrained_vs_det}
If $\sigma=0$, the value function $V$  is a continuous, constrained viscosity solution of equation \eqref{OHJB} on  $[0,\infty)$. Moreover,  $V$ satisfies \eqref{OHJB}  at $x=0$ in the classical sense. 
\end{thm}
\noindent
It was proved in \cite{KLS} that, when $\sigma=0$, the value function is a continuous viscosity solution of \eqref{OHJB} in $(0,\infty)$. Here, we establish classical differentiability of $V$ at $x=0$, as well.
\begin{thm}\label{comparison} Let $0\leq \sigma^2< \rho+2b$. Assume that  
\begin{itemize}
	\item
	$u\in C([0,\infty))$ is a strictly decreasing subsolution of \eqref{OHJB} in $[0, \infty)$, with $Du\leq -\frac{1}{c_1}$,  in the viscosity sense, for some positive constant $c_1$.
	\item
	$v\in C([0,\infty))$ is a strictly decreasing supersolution of \eqref{OHJB} in $(0, \infty)$, such that $v \geq -c_2(1+x^q)$. Here, $c_2$ can be any positive constant and $q$ can be any real number strictly smaller than $|k(\sigma)|$, where $k(0)=+\infty$, and $k(\sigma)$ is the negative root of  \eqref{root2}, for $\sigma>0$.
\end{itemize}
Then, $u \leq v$ in $[0, \infty)$. 
\end{thm}
\begin{rmk}
It follows from \eqref{root2} that when $\sigma^2\in[0,\rho+2b)$, we have $|k(\sigma)|>2$.
\end{rmk}
\noindent
Theorem \ref{comparison} was proved in \cite{KLS, KL} for the stochastic case ($\sigma>0$), but was only stated for $q=2$. Here, we extend this result for $\sigma=0$. Theorem \ref{comparison} is a comparison principle that essentially establishes uniqueness of viscosity solutions in a suitable class. It is used in a decisive manner to establish existence of an optimal control for the deterministic shallow lake problem without any restriction in the parameter space, which is one of the novelties of this article.
\begin{thm}\label{ver}
The deterministic and the stochastic shallow lake problem admit an optimal control, which satisfies:
\begin{equation}\label{opt_control}
    u^*(t)=-\frac{1}{V'(x(t))}, \; t\geq 0
\end{equation}
where $x(\cdot)$ is the solution to \eqref{sldyn} corresponding to this control.
\end{thm}

\subsection{Generalization of the Arrhenius' law}

\noindent
We present now the setting of a general control theory problem to highlight the natural appearance of a noise-dependent potential in control theory problems, and we suggest a methodology to prove that Arrhenius' law is stable under this new model.
\noindent

We assume that there exists a filtered probability space $( \Omega, \mathcal{F},\{ \mathcal{F}_t\}_{t\ge 0},\PP )$
satisfying the usual conditions, and a Brownian motion $W(\cdot)$ defined on that space. Let us also assume that the state dynamics is now described by the following autonomous stochastic differential equation:
\begin{equation}\label{gensde}
\begin{cases}
dx^{\eps}(t)=f(x^{\eps}(t),u(t))dt+\sqrt{2\varepsilon} dW_t & t\geq 0 \\
x^{\eps}(0)=x\geq 0 & 
\end{cases}
\end{equation}
where the function $f\in C^1\left(\mathbb{R}\times U\right)$ and it satisfies:
\begin{equation}
    \begin{cases}
        |f_x|\leq C \\
        |f(x,u)|\leq C(1+|x|+|u|)
    \end{cases}
\end{equation}
Let $G$ be continuous on $\mathbb{R}\times U$, $G(x,\cdot)\in C^1(U)$ and $\rho>0$. We consider the value function of the problem
\begin{equation}
    V_{\eps}(x)=\sup \limits_{u\in\mathfrak{U}}\mathbb{E}_x \left[ \int_0^{\infty} e^{-\rho s} G(x^{\eps}(s),u(s))ds \right]
\end{equation}
where $\mathfrak{U}$ is the set of $\mathcal{F}_t-$adapted, $\mathbb{P}-$a.s. locally integrable processes with values in $U$ satisfying
$$
\mathbb{E}_x \left[ \int_0^{\infty} e^{-\rho s} G(x^{\eps}(s),u(s))ds \right]<\infty
$$
such that the stochastic differential equation \eqref{gensde} has a unique strong solution $x(\cdot).$\\
\noindent
Henceforward, we make the following assumptions:
\begin{asm}\label{asm2}
    \begin{enumerate}[(i)]
        \item The value function $V_{\eps}$ is a (classical) solution to the associated HJB equation:
\begin{equation}\label{genhjb}
-\eps V''_{\eps}-H(x,V'_{\eps})+\rho V_{\eps}=0
\end{equation}
where $H(x,p)=\sup \limits_{u\in U}\{f(x,u)p+G(x,u)\}$.
\item There exists $\eps_0>0$ such that $V_{\eps}$ and $V'_{\eps}$ are uniformly bounded with respect to $\eps<\eps_0$ on every compact subset of $\mathbb{R}$
\item $H(x,p)$ is $C^2$ and $H_{pp}>0$
\item There exists an optimal stationary Markov control policy of the form $$u^*(s)=g\left(x^*(s),V'_{\eps}(x^*(s))\right)=:\bar{u}^*(x^*(s))$$ such that
$$
f_u(x,\bar{u}^*(x))V'_{\eps}(x)+G_u(x,\bar{u}^*(x))=0
$$
and $g$ is a continuous function.
    \end{enumerate}
\end{asm}
\noindent
Under Assumption \ref{asm2}, the optimally controlled system \eqref{gensde} takes the form: \begin{equation}\label{optgensde}
\begin{cases}
	dx^{\eps}(t)=-F'_{\eps}(x^{\eps}(t))dt+\sqrt{2\eps} dW_t & t\geq 0\\
	x^{\eps}(0)=x\geq 0 &
\end{cases}
\end{equation}
where $F'_{\eps}(x)=-f(x,g(x,V'_{\eps}(x)).$
\noindent
Based on Assumption \ref{asm2} (i)-(ii) and stability property of viscosity solutions,
\begin{equation}\label{vslimit}
V_{\eps} \overset{\eps\rightarrow 0}{\rightarrow }V_0 \text{ locally uniformly }
\end{equation}
where $V_0$ is a viscosity solution to the HJB equation \eqref{genhjb} for $\eps=0$. We make the following assumption on function $V_0$.
\begin{asm}\label{asmdif}
The function $V_0$ of \eqref{vslimit} is almost everywhere differentiable.
\end{asm}

\begin{lemma}\label{F0} Assuming \ref{asm2} and \ref{asmdif}, let $\Omega$ be a compact subset of $\mathbb{R}$. Then:
\begin{enumerate}[(i)]
    \item There exists $C=C(\Omega)$ such that $V''_{\eps}(x)\geq C$ for all $x\in \Omega$ and  $\eps<\eps_0$.
    \item The family of functions $\{F_{\eps}\}_{\eps>0}$ converges locally uniformly to a function $F_0$, which is almost everywhere differentiable with $$F'_0(x)=-f(x,g(x,V'_0(x))).$$ 
    \end{enumerate}
\end{lemma}

We are now ready to state our main result:
\begin{thm}\label{arrh}
We assume that the function $F_0$ of Lemma \ref{F0} forms a double well potential with local minima  $x_{\pm}$ and local maximum $x_*,$ with $x_-<x_*<x_+$ and that $F_0$ is $C^1$ close to $x_{\pm},x_*$. Furthermore, we assume that there exist $a,q, M>0$ and $b\in \mathbb{R}$ (independent of $\eps$) such that $F_{\eps}(x)\geq ax^q+b,$ for all $x>M.$  Let $\tau^{\eps}_{x_-}=\inf \{t\geq 0:\; x^{\eps}(t)\leq x_- \}$ be the first hitting time of $x_-$ of the stochastic process $x^{\eps}$ of \eqref{optgensde}. Then the expectation of $\tau^{\eps}_{x_-}$, when $x^{\eps}$  starts  at $x_+$, satisfies
\begin{equation}\label{Law}
	\lim \limits_{\eps \rightarrow 0}\varepsilon\log \EE_{x_+} \left[\tau^{\eps}_{x_-}  \right]=F_0(x_*)-F_0(x_+)
\end{equation}
\end{thm}


We will now state a result which is of general interest and stands on its own. We begin by stating the assumptions of the Lemma.

\begin{asm}\label{asmlemma}
 We assume that the drift $f$ is linear with respect to u, i.e. it is of the form $f(x,u)=a(x)u+b(x)$ and the function $G(x,u)$ is a concave function of $u$ for all $x\in \mathbb{R}.$ We further assume that the noiseless value function $V_0$ is a classical solution to the HJB eq. \eqref{genhjb} for $\varepsilon=0$ everywhere except for a finite number of points $x_0$ whereat it is not differentiable, but there exist the side derivatives $V_0'(x_0(-)), \; V_0'(x_0(+))$.    
\end{asm}
\begin{lemma}\label{genres}
Under Assumption \ref{asmlemma}, if there exists the optimal control, then it is given in the feedback form $u^*(x)=h_x(-a(x)V_0'(x))$, where $h_x$ is the inverse of the function $G_u(x,\cdot),$ and for the drift of the optimally controlled system at $x_0$ we have:
$$
f(x_0, u^*(x_0(-)))\cdot f(x_0, u^*(x_0(+)))\neq 0
$$
\end{lemma}

An interesting application of this result is presented in Lemma \ref{final}(iv) in Section \ref{further} in the case of the shallow lake problem. The fact that the drift of the optimally controlled lake does not disappear at the Skiba point, from the left and from the right, guarantees the existence of the side limits of the second derivative of $V$ at the Skiba point.

The rest of the paper is organised as follows. In section 3, we present the Pontryagin Maximum Principle and describe how the candidate value function is constructed for the shallow lake problem. In section 4, we prove  Theorems \ref{constrained_vs_det}-\ref{ver} which refer to the existence of optimal control. In section 5, we prove Lemma \ref{genres} the generalization of the Arrhenius' law (Lemma \ref{F0} and Theorem \ref{arrh}) and prove that the shallow lake problem satisfies the hypotheses of Theorem \ref{arrh} under suitable assumptions.  Finally, in section 6, we present some further estimates on the value function of the deterministic and stochastic shallow lake problem.

\section{Framework}
\noindent

The Pontryagin Maximum Principle provides necessary conditions for optimality. In our case, the Pontryagin Maximum Principle takes the following form: if $u^*: [0,\infty) \rightarrow (0,\infty)$ is an optimal control  and $x^*(t)$ is the associated optimal trajectory, then there exists a function $p^* (t)$, called the co-state, such that if
\begin{equation}\label{ham}
\tilde{H}(x,u,p)=\ln u-cx^2+p\left( u-bx+r(x) \right)
\end{equation}
then
\begin{enumerate}[1.]
\item $x^*,p^*$ are solutions of the system:
\begin{equation}\label{costate}
	\begin{cases}	\dot{x}=\frac{\theta \tilde{H}}{\theta p}=u-bx+r(x) \\
		\dot{p}=-\frac{\theta \tilde{H}}{\theta x}+\rho p=\left(\rho+b-r'(x)\right)p+2cx
	\end{cases}
\end{equation}
\item $u^*(t)$ maximizes the Hamiltonian $\tilde{H}$ i.e.
\begin{equation}\label{max}
	\max \limits_{u}\tilde{H}(x^*(t),u,p^*(t))=\tilde{H}(x^*(t),u^*(t),p^*(t)) \; \Rightarrow \; u^*(t)=-\frac{1}{p^*(t)}
\end{equation}
\item $p^*(t)$ satisfies the transversality condition \begin{equation}
	\lim \limits_{t\rightarrow \infty} e^{-\rho t}p^*(t)=0
\end{equation}
\end{enumerate}


Due to relation \eqref{max}, there is a one-to-one correspondence between the control $u^*$ and the co-state $p^*$ and the system \eqref{costate} can be rewritten in the state-control form:
\begin{equation}\label{slakes}
\begin{cases}	\dot{x}=u-bx+r(x)=:f(u,x) \\
	\dot{u}=-\left(\rho+b-r'(x)\right)u+2cxu^2=:g(u,x)=2cxu\left(u-g_1(x) \right)
\end{cases}
\end{equation}
The autonomous system \eqref{slakes} is called the \textit{shallow lake system} and its phase curves correspond to potential optimal trajectories of the shallow lake problem. This system  may either have one or multiple equilibria (see  \cite{W03}). In the former case, the equilibrium is a saddle point, while in
the latter there are always two saddle points. The leftmost one is characterised as the oligotrophic steady state of the lake and the rightmost one is called the eutrophic steady state.  In the Appendix of \cite{W03}, it is proved that the only admissible solution curves for optimality are the stable manifolds of the saddle points and three different cases are distinguished.

\begin{itemize}
\item The lake moves towards the oligotrophic steady state regardless its initial pollution level, $x_0$.
\item The lake moves towards the eutrophic steady state regardless its initial pollution level, $x_0$.
\item There exists a threshold value, $x_*$, of the initial pollution level: if $x_0<x_*$, then the lake moves towards the oligotrophic steady state, whereas if $x_0>x_*$, the lake moves towards the eutrophic steady state. The point $x_*$ can either be a repeller or an indifference point (Skiba point). When $x_*$ is a repeller, it is itself a steady state and the resulting policy is everywhere single-valued. On the other hand, indifference points are initial states for which there are two distinct controls corresponding to the same total benefit. One of these controls leads to the oligotrophic steady-state while the other one leads to the eutrophic steady-state. In this case, the resulting policy is everywhere single-valued except for the indifference point, at which it may take two values, see Figure \ref{skibapic}.
\end{itemize}

\begin{figure}[ht]
\centering
\includegraphics[width=0.8\textwidth]{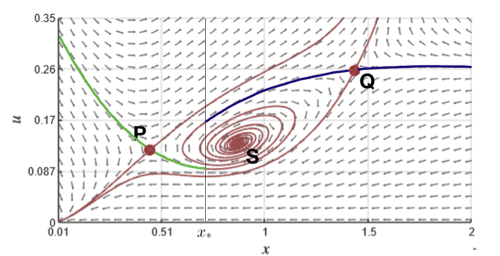}
\caption{Part of the phase plane of the system \eqref{slakes}. The points P, Q are the saddle steady-states, the point S is a vortex and $x_*$ is the Skiba point. The green and the blue curve form the optimal solution, which is everywhere single-valued except for the Skiba point whereat the optimal control may take two values.}
\label{skibapic}
\end{figure}

Based on this analysis, the candidate optimal path $(x^*,u^*)$ suggests a total benefit, called $J_P$, which, following \cite{KosZoh}, is constructed as follows:

Let $(x_0,u_0)$ be a saddle point of \eqref{slakes}. By definition of the saddle points, following \cite{KosZoh}, the total benefit $J_P$, which corresponds to system \eqref{slakes} computed at $x_0$ is given by:

\noindent
$$
    J_P(x_0)=\int \limits_{0}^{\infty} e^{-\rho t} \left( \ln u(t)-cx^2(t) \right)dt =\int \limits_{0}^{\infty} e^{-\rho t} \left( \ln u_0-cx^2_0 \right)dt=\frac{\ln u_0-cx^2_0}{\rho} 
$$

Then the total benefit at any point x can be found via the stable manifold of the corresponding saddle point (see the three cases above) through integration, as follows:

\begin{equation}\label{JP}
J_P(x)=J_P(x_0)+\int \limits_{x_0}^x \frac{dJ_P}{dk}(k)dk=J_P(x_0)+\int \limits_{x_0}^x p(k)dk=J_P(x_0)-\int \limits_{x_0}^x \frac{1}{u(k)}dk
\end{equation}
\noindent
For the second equality, we used that $\displaystyle{\frac{dJ_P}{dx}=p}$ along the trajectories of \eqref{slakes} (for a proof, see the Appendix of \cite{Xep}), while for the third one, we used \eqref{max}. The function $J_P$ will henceforth  serve as a \textit{candidate value function}, and this is how we will refer to it.

\section{Verification Principle}

\subsection{Deterministic Shallow Lake Problem}\label{det_ex}

In this section, we prove Theorems \ref{constrained_vs_det}-\ref{ver} for $\sigma=0$. 
Notice that in the deterministic case, Skiba points appear for certain regions of the parameter space. At these points, the value function $V$ fails to be differentiable and the standard approach to derive the Verification Principle (see \cite{FS} for instance) is not applicable. To prove that the optimal value is attainable, we establish a Comparison Principle (Theorem \ref{comparison}) and we use it to dominate the value function, $V$, by our candidate value function, $J_P$.
\subsubsection{Proof of Theorem \ref{constrained_vs_det}}

\begin{proof}
	In \cite{KLS} it was proved that $V$ is a viscosity solution to \eqref{OHJB} in $(0,\infty)$ for both the deterministic and the stochastic case. Following the lines of proof of Theorem 2.1 \cite{KLS}, what remains to be shown is that for $\sigma=0$, the value function $V$ satisfies (classically) eq. \eqref{OHJB} at $x=0$. Therefore, we will show that $V$ is differentiable at $0$ and
	\begin{equation}\label{dif0}
		\ln (-V'(0))+\rho V(0)+1=0
	\end{equation}
	We will first prove that $\displaystyle{\liminf \limits_{h \rightarrow 0^+} \frac{V(h)-V(0)}{h}\geq -e^{-(\rho V(0)+1)}}$.
	
	\begin{clm} \label{exh}
		There exists $h>0$ such that there exists $ \varepsilon_0>0$ such that for all $ \varepsilon \in (0,\varepsilon_0)$ if  $ u^{\varepsilon}$ is an $\varepsilon-$optimal control, then $\tau^h_{u^{\varepsilon}}:=\inf\{ t\geq 0: x_{\varepsilon}(t)\geq h \}<\infty$, where $x_{\varepsilon}(\cdot)$ is the solution of the \eqref{sldyn} with control $u^{\varepsilon}$ and $x_{\varepsilon}(0)=0$.
	\end{clm}
	\begin{proof}
		By contradiction, let us assume that for any $ h>0$ there exists $\{\varepsilon_n\}_{n\in \NN}$ $(\varepsilon_n \downarrow 0)$ for which there exists $u^n$, $\varepsilon_n-$optimal control with $\tau^h_{u^n}=\infty.$ Then $\forall \; n \in \NN$, if $h_n=1/n,$ $\exists \; \varepsilon_n>0$ $(\varepsilon_n \downarrow 0)$ and $ u^n$ $\varepsilon_n-$optimal control such that $\tau^h_n:=\tau^h_{u^n}=\infty.$ If $x_n=x_{\eps_n}$, then $x_n(t)\leq 1/n \; \forall t\geq 0.$ \\
		Using the elementary inequality  $\displaystyle{\ln u^n \leq \ln A_n +\frac{u^n}{A_n}-1}$ with $A_n=1/n$, we find
		$$
		J(0;u^n)=\int \limits_{0}^{\infty} e^{-\rho t}(\ln u^n(t)-cx^2_n(t))dt\leq \frac{\ln (1/n) }{\rho} + n \int \limits_{0}^{\infty} e^{-\rho t}u^n(t) dt
		$$
			\begin{align*}
				\int \limits_{0}^{\infty} e^{-\rho t}u^n(t) dt&= \int \limits_{0}^{\infty} e^{-\rho t}\dot{x}_n(t)dt+\int \limits_{0}^{\infty} e^{-\rho t}\left( bx_n(t)-r\left(x_n(t)\right) \right)dt \\ 
				&=   \rho \int \limits_{0}^{\infty} e^{-\rho t}x_n(t)dt+\int \limits_{0}^{\infty} e^{-\rho t}\left( bx_n(t)-r\left(x_n(t)\right) \right) dt \\
				&\leq \frac{\rho +b}{\rho n}
			\end{align*}
		Therefore
		$$
		J(0;u^n)\leq  \frac{\ln (1/n) +\rho +b }{\rho} \; \Rightarrow  \; V(0)=\lim \limits_{n\rightarrow \infty}J(0;u^n)=- \infty
		$$
		which is a contradiction.
		
	\end{proof} 
	\begin{clm}\label{posinf}
		$\displaystyle{c(h)= \inf \limits_{0<\varepsilon <\varepsilon_0} \{\tau^h_{u^{\varepsilon}}\}>0}$, where $h$ and  $\{\tau^h_{u^{\varepsilon}}\}_{0<\varepsilon <\varepsilon_0}$ are defined in Claim \ref{exh}. 
	\end{clm}
	\begin{proof}
 
	By contradiction, let us assume that $\displaystyle{ \inf \limits_{0<\varepsilon <\varepsilon_0} \{\tau^h_{u^{\varepsilon}}\}=0}$. Then there exists a  subsequence of stopping times $\{\tau^h_{u^n}\}_{n \in \NN}$ corresponding to a sequence in $ (0,\varepsilon_0)$ $\varepsilon_n \downarrow 0$ such that $\displaystyle{\tau^h_{u^n} \rightarrow 0}$. \\
		Then, since for all  $t\in \left(0,\tau^h_{u^n}\right), n\in \NN,$
		$$
		h=\int \limits_0^{\tau^h_{u^n}}\left( u^n(t)-bx_n(t)+r\left(x_n(t)\right) \right) dt$$  and
		$$0\leq x_n(t) \leq h$$
		
		we have
		\begin{equation}\label{limint}
			\lim \limits_{n\rightarrow \infty} \int \limits_0^{\tau^h_{u^n}}u^n(t) dt=h
		\end{equation}
		Then for $\displaystyle{\phi_n=\frac{1-e^{-\rho \tau^h_{u^n}}}{\rho}}$,  Jensen's inequality gives:
		\begin{equation}
			\begin{split}
				J(0;u^n) & \leq \int \limits_{0}^{\tau^h_{u^n}} e^{-\rho t}(\ln u^n(t)-c{x^2}_n(t))dt + e^{-\rho \tau^h_{u^n}}V(h)\\ &\leq \phi_n \ln \left( \frac{1}{\phi_n} \int \limits_{0}^{\tau^h_{u^n}} e^{-\rho t} u^n(t)dt \right)+V(h)\\
				& \leq \phi_n \ln \left( \frac{1}{\phi_n} \right) +\phi_n \ln \left( \int \limits_{0}^{\tau^h_{u^n}}  u^n(t)dt \right)+V(h)
			\end{split}
		\end{equation}
		Since $\lim\limits_{n\rightarrow \infty} \phi_n=0$, by \eqref{limint}, we get:
		$$
		V(0)=\lim \limits_{n\rightarrow \infty} 	J(0;u^n)\leq V(h)
		$$
		which is a contradiction since $V$ is strictly decreasing (see Proposition 3(ii) \cite{KL}).
	\end{proof}
	\noindent
	We consider now $h,\; u^{\varepsilon}, \tau^h_{u^{\varepsilon}}$ and $x_{\varepsilon}$ as in Claim \ref{exh}. Then for all $\eps \in \left(0, \eps_0\right)$
	$$
	V(0)-\varepsilon<J(0;u^{\eps})<\int \limits_{0}^{\tau^h_{u^{\eps}}} e^{-\rho t}(\ln u^{\eps}(t)-cx^2_{\eps}(t))dt + e^{-\rho \tau^h_{u^{\eps}}}V(h)
	$$
	Using the elementary inequality $\displaystyle{\ln u^{\eps} \leq \ln A +\frac{u^{\eps}}{A}-1, \text{ for }  A>0}$, we\\ find that
	$$
	V(0)-\eps< \frac{\ln A -1}{\rho} \left(1-e^{-\rho\tau^h_{u^{\eps}}}\right) + \frac{1}{A} \int \limits_{0}^{\tau^h_{u^{\eps}}} e^{-\rho t}u^{\eps}(t)dt + e^{-\rho \tau^h_{u^{\eps}}}V(h)
	$$
	Moreover,
	\begin{align*}
	e^{-\rho \tau^h_{u^{\eps}}}h&=\int \limits_0^{\tau^h_{u^{\eps}}} e^{-\rho t} \left( u^{\eps}(t) -(b+\rho)x_{\eps}(t)+ r\left(x_{\eps}(t)\right) \right) dt \\ &\geq \int \limits_0^{\tau^h_{u^{\eps}}} e^{-\rho t}  u^{\eps}(t)dt - \frac{1-e^{ -\rho\tau^h_{ u^{\eps} } }}{\rho}  (b+\rho)h
	\end{align*}
	
	$$
	V(0)-\eps< \left(\ln A -1+\frac{(b+\rho)h}{A}\right) \frac{1-e^{ -\rho\tau^h_{ u^{\eps} } }}{\rho}   + e^{-\rho \tau^h_{u^{\eps}}}V(h) + \frac{he^{-\rho \tau^h_{u^{\eps}}} }{A}
	$$
	$$
	V(0)-V(h) \leq \frac{1-e^{ -\rho\tau^h_{ u^{\eps} } }}{\rho}  \left( \ln A -1+\frac{(b+\rho)h}{A}-\frac{\rho h}{A} -\rho V(h)  \right) +\eps +\frac{h}{A}
	$$
	Choosing $\displaystyle{A=\frac{h}{V(0)-V(h)}},$ we find
	\begin{equation*}
		-\ln \left(- \frac{V(h)-V(0)}{h}\right) -1+b(V(0)-V(h)) -\rho V(0) \geq -\frac{\rho \eps}{1-e^{-\rho \tau^h_{ u^{\eps} } }}
	\end{equation*}
	%
	
	$$
	\frac{V(h)-V(0)}{h} \geq  -\exp \left( \left(- 1-\rho V(0) + b(V(0)-V(h) + \frac{\rho \eps}{1-e^{-\rho c(h) }} \right) \right)
	$$
	where the last inequality follows from Claim \ref{posinf}.
	Letting now $\eps \rightarrow 0^+$ and then $h \rightarrow 0^+$,  we find:
	$$
	\liminf \limits_{h\rightarrow 0^+} \frac{V(h)-V(0)}{h} \geq  -e^{ \left( -1-\rho V(0)  \right)}
	$$ 
	
	Moreover, from Proposition 3(ii) \cite{KL}, we have that 
	$$
	\limsup \limits_{h\rightarrow 0^+} \frac{V(h)-V(0)}{h} \leq  -e^{ \left( -1-\rho V(0)  \right)}
	$$ 
	and this gives \eqref{dif0} and concludes the proof.

\end{proof}

\subsubsection{Proof of Theorem \ref{comparison}}

The proof of Theorem \ref{comparison} is along the lines of the  proof of Theorem 2.2 in \cite{KLS}. Based on Lemma 2.1 \cite{KL}, given a subsolution $u$ and a supersolution $v$ of \eqref{OHJB} satisfying certain assumptions, the difference $u-v$ is a subsolution of the corresponding linearized equation. The difference with the proof presented in \cite{KLS} is that we need to add a perturbation term in the linearized equation in order to be able to find a smooth solution and to compare it with the difference $u-v$.

\begin{proof}
  Let $\eta>0$ sufficiently small. We  consider the ordinary differential equation 
	\begin{equation}\label{ode11}
		\rho w +\big(bx-(a+c^*)\big)w'-\frac{1}{2}\eta x^2w''=0,
	\end{equation}
	which has a solution of the form 
	\begin{equation}\label{sol1}
		w(x)=x^{-k}\mathcal{J}(\frac{2a+2c^*}{\eta x}),
	\end{equation}
	where $k$ is a root of  
	\begin{equation}\label{root2}
		k^2+\Big(1+\frac{2b}{\eta}\Big)k-\frac{2\rho}{\eta}=0
	\end{equation}
	and   $\mathcal{J}$ a solution of the degenerate hypergeometric equation 
	\begin{equation}\label{confluent1}
		xy''+(\tilde{b}-x)y'-\tilde{a}y=0
	\end{equation}
	with $\tilde{a}=k$ and $\tilde{b}=2(k+1+b/\eta)$.
	
 We choose $k$ to be the negative root of \eqref{root2}. 
 We further choose $\mathcal{J}$ to be the Tricomi solution of \eqref{confluent1} which satisfies
	\[
	\mathcal{J}(0)>0\qquad\text{and} \qquad \mathcal{J}(x)=x^{-k}\big(1+\frac{2\rho}{\eta x}+o(x^{-1})\big)\quad\text{as }x\to\infty.
	\]
	
	With this choice, the function $w$ defined in \eqref{sol1} for $x>0$ and by continuity at $x=0$, satisfies $w(0), w'(0) >0$ and $w(x)\sim \mathcal{J}(0) x^{-k}$, as $x\to\infty$. We choose $\eta>0$ sufficiently small so that $k<-q.$
	
	Note that $w$ is increasing in $[0,\infty)$ since it would otherwise have a positive local maximum and this is impossible by \eqref{ode11}. 
	
	Set now $\psi=u-v$ and consider $\epsilon>0$. Since $\psi-\epsilon w<0$  in a neighborhood of infinity, there exists $x^\epsilon \in [0,\infty)$ such that
	\[
	\max_{x\ge 0} \big(\psi(x)-\epsilon w(x)\big)=\psi(x^\epsilon)-\epsilon w(x^\epsilon).
	\]
	
	By Lemma 2.1 \cite{KL} (with $\sigma=0$), $\psi $ is a subsolution of 
	$$
	\rho \psi +bx D\psi-(a+c^*)|D\psi|=0 \; \in [0,\infty)
	$$
	We now use $\epsilon w$ as a test function to find that
	\begin{align*}
		0&\ge \rho \psi({x}^\epsilon) +\epsilon b{x}^\epsilon w'({x}^\epsilon)-\epsilon (a+c^*)|w'({x}^\epsilon))| \\ & \overset{\eqref{ode11}}{=}\rho \left( \psi(x^{\varepsilon}) - \varepsilon w(x^{\varepsilon})\right)+\frac{\eta}{2} (x^{\varepsilon})^2 \varepsilon w''(x^{\varepsilon})
	\end{align*}
	$$
	\rho \left( \psi(x^{\varepsilon}) - \varepsilon w(x^{\varepsilon})\right)\leq -\frac{1}{2}\eta (x^{\varepsilon})^2 \varepsilon w''(x^{\varepsilon})
	$$
	Note that the function $g(x)=\frac{1}{2}\eta x^2  w''(x)
	$ is bounded from below because from \eqref{ode11} for $x>(a+c^*)/b$ we have
	$g(x)= \rho w(x)+ (bx-(a+c^*))w'(x)\geq \rho w(0)
	$, since $w$ is increasing in $[0,\infty)$. Therefore,
	$$
	\rho \left( \psi(x) - \varepsilon w(x)\right)\leq - \varepsilon \inf \limits_{[0,\infty)}g \; \;  \text{for all} \; x\in [0,\infty)
	$$
	Since $\varepsilon$ is arbitrary, this proves the claim.
\end{proof}

\subsubsection{Proof of Theorem 3}

With the Comparison Principle at our disposal, we proceed to show that our candidate value function satisfies the assumptions made for the supersolution $v$ in Theorem \ref{comparison}.

\begin{propos}\label{propofJp}
	Let $J_P$ be the candidate value function of \eqref{JP}. Then,
	\begin{enumerate}[i.]
	\item $J_P$ is decreasing.
		\item $J_P$ is a viscosity solution to \eqref{OHJB} with $\sigma=0$, in $(0,\infty).$
		\item There exists $c_2>0$, such that $J_P(x)\geq -c_2(1+x^2)$,  for all $x\geq 0$. 
	\end{enumerate}
\end{propos}

\begin{proof}
	\begin{enumerate}[i.]
	\item Along the stable manifold  $\displaystyle{\frac{dJ_P}{dx}=-\frac{1}{u}}<0$. Therefore, $J_P$ is decreasing. Moreover, $\dot{u}=g(x,u)<0$ for $x$ close to zero, which implies that  $J_P(0)=\lim \limits_{x\rightarrow 0} J_P(x)<\infty.$ 
\item The total benefit $J_P(x)$  of \eqref{JP}, for $x$ different from the Skiba point (if there exists such a point), is the classical solution to  eq. \eqref{OHJB}  constructed by the method of characteristics. In the case of a Skiba point, it was proved in \cite{KosZoh} that $J_P$ is also a viscosity solution to \eqref{OHJB} at the Skiba point. 
		\item	Let $x>\frac{a+1}{b}$. Since $J_P$ is a classical solution to \eqref{OHJB} at $x$ and $J_P'<0$, we have that 
		\begin{align*}
		\rho J_P(x)&=\sup \limits_{u>0} \bigg\{ (u-bx+r(x))J'_P(x)+\ln u -x^2 \bigg\} \\ &\geq \ln \left(bx-r(x)\right) -x^2\geq -x^2
		\end{align*}
		By continuity of $J_p$ in $[0,\infty)$, we conclude the proof. 
	\end{enumerate}
\end{proof} 

We have now collected all the key ingredients to establish our main existence result.

\noindent
\textit{Proof of Theorem \ref{ver}($\sigma=0$):}
	It is a direct consequence of Theorems \ref{constrained_vs_det}, \ref{comparison}, Proposition \ref{propofJp} and Proposition 3 in \cite{KL} that $V\leq J_P$. Hence, the control suggested by the Pontryagin Maximun Principle is indeed optimal  and it is given in the feedback form as $u^*(t)=u(x(t))=-\frac{1}{ V'(x(t))}.$ \qed

\subsection{Stochastic Shallow Lake Problem}

In this section, we prove existence in the presence of noise (Theorem \ref{ver} for $\sigma>0$). In \cite{KLS, KL} it was proved that $V$ is a viscosity solution of \eqref{OHJB} and from classical results for uniformly elliptic operators, it follows that $V$ is actually a classical solution of \eqref{OHJB} in $(0,\infty)$.  In fact, it can be proved that $V$ is actually two times differentiable at $x=0$ and in this way, it follows that $V$ is a classical solution of \eqref{OHJB} in $[0,\infty).$ This result is stated in Proposition \ref{constrained_vs}.

\begin{propos}\label{constrained_vs}
	If $0<\sigma^2<\rho+2b,$ the value function $V$ is a classical solution of the equation \eqref{OHJB} in $[0,\infty)$ and $$V''(0)=-(\rho+b-r'(0))\left(V'(0)\right)^2$$
\end{propos}

\begin{proof}
	Based on Theorem 2.1 of \cite{KL}, the value function $V$ is a constrained viscosity solution of \eqref{OHJB} in $[0,\infty)$. What remains to be shown is that $V$ satisfies (OHJB) at $x=0$ in the classical sense. Based on Corollary 2.1 \cite{KL}, it suffices to show that $V$ is two times differentiable at $x=0$ and $C^1([0,\infty)).$ We will first show that $V\in C^1([0,\infty))$.\\
	Let $0<x<b$. From Propositions 3(iii),4 \cite{KL}, we have that:
	$$
	-\Phi_{\sigma}(x)\leq V'(x) \leq -c(x)
	$$
	Taking $x\rightarrow 0$, we have that $\lim \limits_{x\rightarrow 0}V'(x)=e^{- \rho V(0)-1 } = V'(0)$.\\
	\noindent 
	Therefore, $V\in C^1([0,\infty))$.\\
	Now, we proceed with the second derivative of $V$.\\
		From eq. \eqref{OHJB}, we have that for $x>0$:
	\begin{equation}\label{v2}	
		V''(x)=\frac{2}{\sigma^2} \left[ \frac{\rho V(x)+\left(bx-r(x)\right)V'(x)+\ln (-V'(x))+1+cx^2}{x^2} \right] 
	\end{equation}
	Setting $Q(x)=\rho V(x)+(bx-r(x))V'(x)+\ln (-V'(x))+1+cx^2$, we have that
	
	$$Q'(x)=(\rho+b-r'(x))V'(x)+2cx+(bx-r(x))V''(x)+\frac{V''(x)}{V'(x)} \; \Rightarrow $$
	
	\begin{align*}
	Q'(x)=-\frac{2}{\sigma^2x^2}\left( - \frac{1}{V'(x)}+r(x)-bx \right) Q(x) &+ (\rho +b-r'(x))V'(x)+2cx
	\end{align*}
	If we denote
	\[
	\alpha(x)=-\frac{2}{\sigma^2x^2}\left( - \frac{1}{V'(x)}+r(x)-bx \right) \qquad\text{and}\qquad \beta(x)=(\rho +b-r'(x))V'(x)+2cx,
	\]
	then for any $\varepsilon >0$ we have
%
	$$
	Q(x)=Q(\varepsilon)e^{\int \limits_{\varepsilon}^{x}\alpha (s)ds}+\int_{\varepsilon}^{x}\beta(s)e^{\int_{s}^{x}\alpha (t)dt}ds
	$$
	
	Since 
	$\displaystyle{
		\lim \limits_{\varepsilon \rightarrow 0^+}\int \limits_{\varepsilon}^x a(s)ds=-\infty}
	$
	and $\displaystyle{\lim \limits_{\varepsilon \rightarrow 0}Q(\varepsilon)=0}$, 
	
	$$
	\frac{1}{2}\sigma^2V''(x)=\frac{Q(x)}{x^2}=\frac{1}{x^2} {\int \limits_{0}^{x}\beta(s)e^{\int \limits_{s}^{x}\alpha (t)dt}ds}
	$$
	
	Let $0<\eta< \min\{-\beta(0), -1/V'(0)\}$. There exists $\varepsilon>0$ such that $ \forall s\in[0,\varepsilon]$
	
	$$\begin{cases}
		\beta(0)-\eta<\beta(s)<\beta(0)+\eta 
		\\
		\left( \frac{1}{V'(0)}-\eta \right)\frac{2}{\sigma^2 s^2} <\alpha(s)<  \left( \frac{1}{V'(0)}+\eta \right)\frac{2}{\sigma^2 s^2} 
	\end{cases}
	$$
	
	Then, for $x\in[0,\varepsilon]$,
	\begin{equation*}
	\frac{Q(x)}{x^2} \leq \frac{(\beta(0)+\eta)\int \limits_{0}^{x}e^{\int \limits_{s}^{x}  ( \frac{1}{V'(0)}-\eta )\frac{2dt}{\sigma^2 t^2}  }ds}{x^2} 
	=(\beta(0)+\eta)\int_0^\infty e^{ ( \frac{1}{V'(0)}-\eta )\frac{2t}{\sigma^2}}\frac{dt}{(1+tx)^2}.
	\end{equation*}
	Let now $x\to 0$, then $\eta\to 0$ to get
	\begin{equation}\label{up}
		\limsup \limits_{x\rightarrow 0}\frac{Q(x)}{x^2}\leq -\beta(0)V'(0)\sigma^2/2.
	\end{equation}
	Similarly, 
	\begin{equation}\label{bel}
		\liminf \limits_{x\rightarrow 0}\frac{Q(x)}{x^2}\geq -\beta(0)V'(0)\sigma^2/2.
	\end{equation}
	Therefore, from \eqref{v2}, \eqref{up}, \eqref{bel}, we have
	
	$$
	\lim\limits_{x\rightarrow 0}V''(x)=-(\rho+b-r'(0))(V'(0))^2
	$$
	Since $V\in \mathbb{C}^1[0,\infty) \cap \mathbb{C}^2(0,\infty)$, the assertion follows. 
\end{proof}
The elliptic regularity of the value function in the presence of noise permits the adoption of the usual methodology in order to prove the existence of the optimal control. In this direction, we follow the steps described in \cite{FS}. The fact that our (candidate) optimal control is not bounded away from zero and therefore the fact that the logarithmic term in the total benefit may be $-\infty$ demands some extra technical manipulations.

\noindent
\textit{Proof of Theorem \ref{ver}($\sigma>0$):}
		Let $x(t)$ be the solution  of the sde
	$$\begin{cases} dx(t)=f(x(t),-\frac{1}{V'(x(t))})dt+\sigma x(t) dW_t \\
		x(0)=x \end{cases}$$\\
	and  $u(t)=-\frac{1}{V'(x(t))}$ the corresponding control.\\
	We apply It$\hat{o}$'s Rule to the stochastic process $g(t,x(t))=e^{-\rho t}V(x(t))$ and we find for $t\geq 0$:

 \begin{align*} V(x)&=e^{-\rho t}V(x(t))+\int \limits_{0}^t e^{-\rho s} \left(\rho V(x(s))-(r(x(s)))-bx(s)))V'(x(s)))+1\right)ds \\ &\hspace{10mm}-\int \limits_{0}^t \frac{1}{2}V''(x(s)))\sigma^2 x^2(s) ds-\int \limits_{0}^{t}e^{-\rho s}\sigma x(s)) V'(x(s)))dW_s \\ &\overset{\eqref{OHJB}}{=}e^{-\rho t}V(x(t))+\int \limits_{0}^t e^{-\rho s} (\ln (u(s))-c x^2(s))ds-\int \limits_{0}^{t}e^{-\rho s}\sigma x(s) V'(x(s))dW_s
\end{align*} 
 

We also consider the stopping times $\theta_n=\inf \{t\geq 0: \; |x(t)-x|\geq n\} \wedge n$ and taking expected values in the above relation, we get that, for all $ n \in \mathbb{N}$:
$$\displaystyle{V(x)=\EE \left[e^{-\rho \theta_n}V(x(\theta_n)) \right]+\EE \left[ \int \limits_{0}^{\theta_n} e^{-\rho s} (\ln (u(s))-cx^2(s))ds  \right]}$$

\begin{itemize}
	\item Since $V$ is decreasing, we have that $$\limsup \limits_{n\rightarrow \infty} \EE \left[e^{-\rho \theta_n}V(x(\theta_n)) \right] \leq  \lim \limits_{n\rightarrow \infty} \EE \left[e^{-\rho \theta_n} \right] V(0) =0$$
	
	\item Since $V'(x)$ is bounded from above by a negative quantity, we have that $u(s)$ is also bounded from above by a constant, say $C$. Thus, we have that:
	
	$$\lim \limits_{n\rightarrow \infty} \EE \left[ \int \limits_{0}^{\theta_n} e^{-\rho s} (C-\ln (u(s)))ds  \right]= \EE \left[ \int \limits_{0}^{\infty} e^{-\rho s} (C-\ln (u(s)))ds  \right]$$ by monotone convergence theorem. Therefore, $$\lim \limits_{n\rightarrow \infty} \EE \left[ \int \limits_{0}^{\theta_n} e^{-\rho s} \ln (u(s))ds  \right]= \EE \left[ \int \limits_{0}^{\infty} e^{-\rho s} \ln (u(s))ds  \right]$$
	
	\item Regarding the last term, we also have that
	$$\lim \limits_{n\rightarrow \infty} \EE \left[ \int \limits_{0}^{\theta_n} e^{-\rho s}x^2(s)ds  \right]= \EE \left[ \int \limits_{0}^{\infty} e^{-\rho s}x^2(s)ds  \right]$$ by monotone convergence theorem. 
	
\end{itemize}

Thus, we have that
$$V(x)\leq \EE \left[ \int \limits_{0}^{\infty} e^{-\rho s} (\ln (u(s))-cx^2(s))ds  \right] =J(x;u)\leq V(x)$$
and this concludes the proof.  \qed

\begin{rmk}
The optimal control, written in feedback form, $u^*(x)=-\frac{1}{V'(x)}$, for $\sigma>0$, is obviously a bounded and locally Lipschitz function. Therefore, the problem \eqref{sldyn} has a unique strong solution $x(\cdot)$ (see Theorem 3.4 \cite{Mao}) and the optimal control $u^*$ is admissible. Furthermore, the admissibility of the optimal control in the deterministic case is an immediate consequence, from the way it was constructed, since it is located on the stable manifold of the Pontryagin system of the lake \eqref{slakes}.
\end{rmk}


\section{Generalization of the Arrhenius Law}

In this section, we prove Lemma \ref{F0}, Theorem \ref{arrh} and Lemma \ref{genres}. In the following, in subsection \ref{application}, we consider the shallow lake problem as an application to the described methodology.

\subsection{Proof of Lemma \ref{F0}}
\begin{enumerate}[(i)]
\item We will prove this claim following the lines of proof of Lemma 3.1 in \cite{FlS}.
		Let $m=\inf\{\Omega\}>-\infty$, $M=\sup\{\Omega\}<\infty$. If $l=\max\{|m|,|M|\}$, we consider a function $\phi \in C^{\infty}(\mathbb{R})$ such that
		$$	
		\begin{array}{rl}
			\phi(x)=1 & \text{if } \; x\in \Omega \\
			\phi(x)=0 & \text{if } \; x<-2l \text{ or } x>2l \\	
			\frac{\left(\phi'(x)\right)^2}{\phi(x)}\leq C_l & \text{on supp } \phi\\
		\end{array}
		$$	
		For brevity, we write $u=V_{\eps}$, $u_1=V'_{\eps}$	and $u_{11}=V''_{\eps}$. Differentiating twice eq. \eqref{genhjb} with respect to $x$, we obtain:
		$$
		-\eps u''_{11}-H_{xx}-2H_{px}u_{11}-H_{pp}\left(u_{11}\right)^2-H_pu'_{11}+\rho u_{11}=0
		$$
		Next we define $w=\phi u_{11}$ and compute (on $\phi>0$):\\
  \vspace{5mm}
		$
		-\eps w''-H_p(x,u_1)w'+\rho w+2\eps \frac{\phi'}{\phi}w'
		$\\
		$
=\phi H_{pp}(x,u_1)\left(u_{11}\right)^2+\phi H_{xx}+2\phi H_{px} u_{11} +\eps \left(\frac{2(\phi')^2)}{\phi}-\phi''  \right) u_{11}+H_p\phi' u_{11} 
  $
  
		Let $x_0$ be a point in ($\phi>0$) at which $w$ attains a negative minimum. If for some value of $\eps$, $w$ is non-negative on ($\phi>0$), we can conclude that $V''_{\eps}(x)\geq 0$ $\forall x\in \Omega$. Then $w'(x_0)=0,$ $w(x_0)\leq 0$ and $w''(x_0)\geq 0$. 
		Moreover, by Assumption \ref{asm2}(iii) there exists  $\eta_0>0$ such that $H_{pp}(x,u_1)>\eta_0>0$ for all $x\in \Omega$. \\
		Therefore, we have at $x=x_0$
		$$
  \eta_0 \left(w(x_0)\right)^2\leq -\phi^2 H_{xx}+ \left( -2\phi H_{px}  -\eps C_l + \eps \phi'' -H_p\phi'\right) w(x_0)
		$$
		\begin{equation*}
				\Rightarrow \eta_0 w^2(x_0)\leq 
				 A+B w(x_0) \; \Rightarrow \; w(x_0)\geq \tilde{C}
		\end{equation*}
		where $\tilde{C}$ constant independent of $\eps$ (and $x_0).$ Here, we used the fact that since $u_{1}=V'_{\eps}$ is bounded on supp $\phi$ by Assumption \ref{asm2}(ii) (uniformly with respect to $\eps$), each partial derivative of $H$ is bounded at $(x,u_1(x))$ for $x\in$ supp $\phi$.\\
		Therefore $V''_{\eps}(x)\geq w(x_0)\geq \tilde{C}$, $\forall x\in\Omega,\; \eps\leq \eps_0$. \hfill$\Box$\
		\item Since $V''_{\eps}$ is locally uniformly bounded with respect to $\eps$ and $V_0$ is almost everywhere differentiable, it follows (see Theorem 3.2 (i) \cite{FlS}) that $V'_{\eps} \rightarrow V'_0$  almost everywhere. Therefore, the continuity of $g$ and the boundedness of $V'_{\eps}$ on compact sets imply, by bounded convergence theorem, that $F_{\eps}$ converges locally uniformly to a function $F_0$, which is almost everywhere differentiable with $$F'_0(x)=-f(x,g(x,V'_0(x))).$$  \hfill$\Box$\
\end{enumerate}


\subsection{Proof of Theorem \ref{arrh}}
	
	For every $\eps>0,$ we know that the function $h(x)=\EE_{x} \left[\tau^{\eps}_{x_-}\right]$ solves the Poisson problem:
    $$
    \begin{cases} \mathcal{L}h(x)=-1, & x>x_- \\
       h(x) =0, & x\leq x_-\end{cases}
    $$
    where $\mathcal{L}$ is the generator of the process $x^{\varepsilon}$ in \eqref{optgensde}. That is
     $$
    \begin{cases} 
       \left( \varepsilon h''(x)-F_{\varepsilon}'(x)h'(x)\right)=-1, & x>x1 \\
       h(x) =0, & x\leq x_-\end{cases}
    $$
    This problem is solved explicitly and for $x=x_+$ takes the form:
	$$
	\EE_{x_+} \left[\tau^{\eps}_{x_-} \right]= \frac{1}{\eps} \int \limits_{x_-}^{x_+} \int \limits_{z}^{\infty} \exp \left(\frac{F_{\eps}(z)-F_{\eps}(y)}{\eps}\right) dy dz
	$$
	
	We denote by $D$ the area of integration. Notice that in $D$, the function $(z,y)\mapsto F_0(z)-F_0(y)$ attains its maximum at $(x_*,x_+)$ and for $A>x_+$, let us consider the compact set $D_1=D\cap \left((-\infty, A]\times \mathbb{R}\right)$ which contains the point $(x_*,x_+)$ and the unbounded set $D_2=D\setminus D_1.$ Moreover, denote by $I_1(\eps),I_2 (\eps)$ the integral of $\exp \left(\frac{F_{\eps}(z)-F_{\eps}(y)}{\eps}\right)$ over $D_1$ and $D_2$, respectively, to get
	\begin{equation}\label{I1I2}
		\iint_{D} e^{\frac{1}{\eps}\left(F_{\eps}(z)-F_{\eps}(y)\right) }dy dz = I_1(\eps)+I_2(\eps)
	\end{equation}
	Since $D_1$ is compact and $F_{\eps}$ converges uniformly on compact sets to $F_0,$ there exists a non-negative function $\lambda (\eps)$, with $\lim \limits_{\eps \rightarrow 0} \lambda (\eps)=0,$ such that $$|(F_{\eps}(z)-F_{\eps}(y))-(F_{0}(z)-F_{0}(y)) |\leq \lambda(\eps) \; \forall (y,z)\in D_1.$$ Therefore,
		$$\displaystyle{e^{-\frac{\lambda(\eps)}{\eps}}
		\iint_{D_1} e^{\frac{1}{\eps}\left(F_{0}(z)-F_{0}(y)\right) }dy dz \leq I_1(\eps) \leq \iint_{D_1} e^{\frac{1}{\eps}\left(F_{0}(z)-F_{0}(y)\right) }dy dz e^{\frac{\lambda(\eps)}{\eps}}}
		$$
		Which gives from standard Laplace asymptotics that, 
		\begin{equation}\label{I1s}
			\lim \limits_{\eps \rightarrow 0} \eps \log I_1 (\eps) = F_{0}(x_*)-F_{0}(x_+)
		\end{equation}
		 Since $F_{\eps}$ converges uniformly to $F_0$ on $[x_-, x_+]$, there exists $\eps_1>0$ such that $F_{\eps}(z)<F_0(z)+1\leq F_0(x_*)+1, \; \forall z\in [x_-,x_+], \; \forall \eps <\eps_0.$ Thus,
		$$I_2(\eps)=\int \limits_{x_-}^{x_+} \int \limits_{A}^{\infty}e^{\frac{1}{\eps}\left(F_{\eps}(z)-F_{\eps}(y)\right) }dy dz \leq  e^{\frac{1}{\eps} (F_0(x_*)+1)}(x_+-x_-) \int \limits_{A}^{\infty}e^{-\frac{1}{\eps}(ay^q+b) }dy $$
  which based on bounds of the upper incomplete Gamma function (see \cite{Pin}), gives
		
		\begin{equation}\label{I2s}
			\limsup \limits_{\sigma \rightarrow 0}\eps \log I_2(\eps) \leq  F_0(x_*)+1 -b -aA^q
		\end{equation}
		By choosing $A$ sufficiently large (such that $1-b-aA^q<-F_0(x_+)$), from relations \eqref{I1I2}, \eqref{I1s}, \eqref{I2s} we find that
		$$
		\lim \limits_{\eps \rightarrow 0} \eps\log   \iint_{D} \exp \left(\frac{F_{\eps}(z)-F_{\eps}(y)}{\eps}\right) dy dz = F_0(x_*)-F_0(x_+)
		$$
		which concludes the proof.  \hfill$\Box$\

\subsection{Proof of Lemma \ref{asmlemma}}

Since the optimal control exists, it should satisfy the following relation:
\[
a(x)V_0'(x)+G_u(x,u)=0 
\]
which gives $u^*(x)=h_x(-a(x)V_0'(x)).$ By contradiction, let us assume that 
	\begin{equation}\label{sbllimgen}
			f(x_0,u^*(x_0(-)))=a(x_0)h_{x_0}\left(-a(x_0)V_0'(x_0(-))\right)+b(x_0)=0,	\end{equation}
   and without loss of generality we assume that $V_0'(x_0(-))<V_0'(x_0(+)).$
		Taking $x\rightarrow x_0(-)$ in HJB, by \eqref{sbllimgen}, we find:
		\begin{equation}\label{sblhjbgen}
		\rho V(x_0)=G\left(x_0,h_{x_0}(-a(x_0)V_0'(x_0(-)))\right)
		\end{equation}
		Taking now $x\rightarrow x_0(+)$ in HJB and substituting \eqref{sblhjbgen} and \eqref{sbllimgen}, we find:
		\begin{multline*}
	G\left(x_0,h_{x_0}(-a(x_0)V_0'(x_0(-)))\right)-G\left(x_0,h_{x_0}(-a(x_0)V_0'(x_0(+)))\right)= \\-a(x_0)\Big(h_{x_0}(-a(x_0)V_0'(x_0(-)))-h_{x_0}(-a(x_0)V_0'(x_0(+)))  \Big)V'(x_0(+))
		\end{multline*}
		which by Mean Value Theorem and the monotonicity of $h_{x_0}$, gives
  \begin{equation}\label{mvt}
  G_u(x_0,h_{x_0}(-a(x_0)y)=-a(x_0)V'(x_0(+)))
\end{equation}	
 for some $y \in (V'(x_0(-)), V'(x_0(+)))$, 
		which is a contradiction, since relation \eqref{mvt} implies that $y=V'(x_0(+))$. \hfill$\Box$\

\subsection{Application to the shallow lake problem}\label{application}

The metastable behaviour of shallow lakes is naturally observed and in mathematical terms corresponds to a system with two local equilibrium points and a Skiba point. In the presence of noise, the system moves from the oligotrophic state to the eutrophic state and vice versa see Figure \ref{troxies}. Furthermore, the markovian nature of the optimal control leads to a system with a noise-dependent drift function. Therefore, the shallow lake problem is offered as a suitable application of our result, Theorem \ref{arrh}. In Figure \ref{potential}, the double-well potential of the deterministic shallow lake problem is depicted, indicating the height of the barrier that the process has to overcome in order to get to the first well. This is identified with the constant on the right term of the Arrhenius' law.

If we apply the transformation $y(t)=\log x(t)$ to the process $x(t)$ of the stochastic shallow lake problem \eqref{sldyn}, we find by  It\^o's rule: 
\begin{equation}\label{con_dif}
 \begin{cases}
		dy(t)=\left(e^{-y(t)}u(t) -b +r\left(e^{y(t)}\right)e^{-y(t)}  -\frac{\sigma^2}{2}\right)dt+\sigma dW_t \\
        y(0)=y_0
\end{cases}
	\end{equation}	
\noindent
Therefore, the dynamics of the shallow lake problem is described in terms of \eqref{gensde}.  We can now consider the value function of the shallow lake problem in terms of the process $y(t)$, i.e.

$$
\tilde{V}(y)=\sup \limits_{u\in \mathfrak{U}} \mathbb{E}_{y} \left[ \int \limits_0^{\infty} e^{-\rho t} \left( \ln u(t) -c e^{2y(t)}\right)dt \right]=V(e^{y})
$$
\noindent
Obviously, it follows from Proposition \ref{constrained_vs} that $\tilde{V}$ is a classical solution in $\mathbb{R}$ of equation
\begin{equation}\label{OHJBlog}
    -\frac{1}{2}\sigma^2 \tilde{V}''-\tilde{H}(x,\tilde{V}')+\rho \tilde{V}=0
\end{equation}
 where $\tilde{H}(x,p)=\left(r(e^x)e^{-x}-b-\frac{\sigma^2}{2}\right)p -\ln(-p)+x-1-ce^{2x}$ and the optimally controlled system \eqref{con_dif} takes the form:
 \begin{equation}\label{optsde}
      \begin{cases}
		dy^{\sigma}(t)=-F'_{\sigma}(y^{\sigma}(t))dt+\sigma dW_t \\
        y^{\sigma}(0)=y_0
\end{cases}
 \end{equation}
where $$F'_{\sigma}(y)=\frac{1}{\tilde{V}'_{\sigma}(y)} +b -r\left(e^{y}\right)e^{-y}  +\frac{\sigma^2}{2}$$
\noindent
Obviously, our verification argument for the deterministic shallow lake problem implies that $F_0$ is $C^1\left( \mathbb{R}\setminus \{x_0\} \right)$. The next two Lemmata show that the shallow lake problem, under the assumption of two saddle equilibrium points and one Skiba point, satisfies the hypothesis of Theorem \ref{arrh}. Therefore, the Arrhenius law is stable under our model.

\begin{figure}
				\centering 
				\begin{subfigure}{0.5\textwidth}
					\includegraphics[width=1\textwidth]{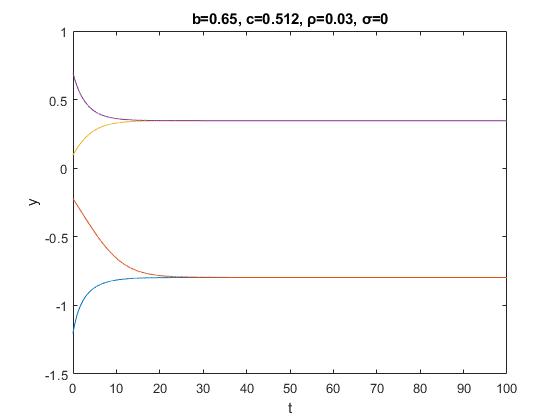}
					\caption{}
					\label{troxies0}
				\end{subfigure}\hfil
				\begin{subfigure}{0.5\textwidth}
					\includegraphics[width=1\textwidth]{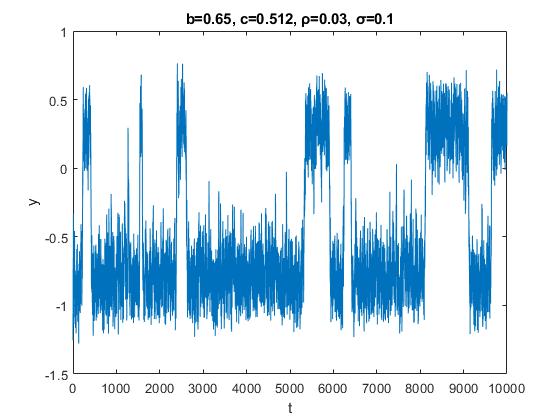}
					\caption{}
					\label{troxies01}
				\end{subfigure}
				\vspace{-1\baselineskip}
				\caption{Left: The paths of the optimally controlled lake (deterministic case) for different initial positions. Right: One simulated path of the optimally controlled lake (stochastic case) with two stochastic attractors.}	
				\label{troxies}
\end{figure}

\begin{figure}
    \centering
    \includegraphics[width=0.6\textwidth]{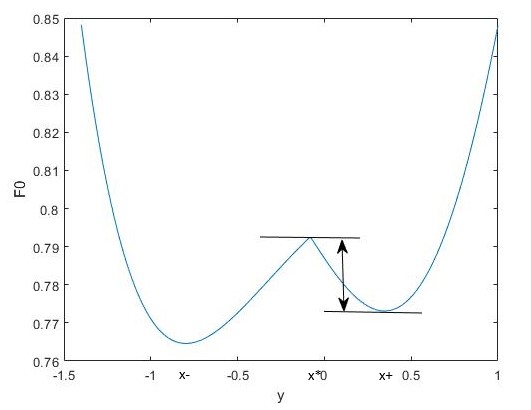}
    \caption{The potential $F_0$ of the deterministic optimally controlled shallow lake when $(b,c,\rho,\sigma)=(0.65,0.512,0.03,0)$}
    \label{potential}
\end{figure}

\begin{lemma}\label{CoOfDer}
	We assume that the recycling rate function $r$, in addition to Assumption \ref{r},  satisfies also that $r(x) <bx$ $\forall \; x>0.$ Let $\Omega \subset \mathbb{R}$ compact and $\sigma_0<\sqrt{\rho}$. Then there exists $C=C(\Omega, \sigma_0)>0$ such that $ |\tilde{V}'_{\sigma}(x)|\leq C$ for all $ x\in \Omega,\;  \sigma\leq \sigma_0$.
	
\end{lemma}
\begin{proof}
Since $\tilde{V}'_{\sigma}(x)=V'_{\sigma}(e^{x})e^{x}$, Proposition 3(ii) \cite{KL} implies that there exists $C>0$ (independent of $\sigma$) such that $\tilde{V}'(x)\leq -C<0$ for all $x\in \Omega.$ Therefore, it suffices to show that there exists a continuous function $\Phi: \; (0,\infty)\rightarrow (0,\infty) $ independent of $\sigma$ such that $V'_{\sigma}(x)\geq -\Phi(x)$ for all $x>0.$ \\
\noindent
 Let $0<x_1<x$ and for $d>0$ consider a control $u\in \mathfrak{U}$ which equals $d$ up to time $\tau_d=\inf \{t\geq 0: \; x(t)\leq x_1 \}$ where $x(\cdot)$ is the solution of \eqref{sldyn} with control $u$ and $x(0)=x.$ Then
\begin{align*}
		V_{\sigma}(x)&\geq \EE \left[ \int_{0}^{\tau_d} e^{-\rho t}(\ln u(t)-cx^2(t))dt \right]+\EE \left[ e^{-\rho \tau_d} \right] V_{\sigma}(x_1)\\
  &= \ln d\cdot \EE  \left[ \int_{0}^{\tau_d} e^{-\rho t}dt  \right] - c \EE  \left[ \int_{0}^{\tau_d} e^{-\rho t} x^2(t)dt  \right] +\EE \left[ e^{-\rho \tau_d} \right] V_{\sigma}(x_1) 
\end{align*}
Based on Propositions 2(iii) and 3(ii) in \cite{KL}, $V_{\sigma}(x)\leq \frac{1}{\rho} \ln \left( \frac{b+\rho}{\sqrt{2 e c}} \right)=:D $ for all $\sigma \leq \sigma_0,$ $x\geq 0$. Thus, we find:
\begin{equation}\label{bVx}
			V_{\sigma}(x)-V_{\sigma}(x_1) \geq (\ln d-\rho D ) \EE \left[ \int_{0}^{\tau_d} e^{-\rho t}dt \right]- \EE \left[ \int_{0}^{\tau_d} e^{-\rho t}c x^2(t)dt \right]
\end{equation}
Applying It$\hat{o}$'s rule to the semimartingale $Y_t=e^{-\rho t}x^2(t)$, then the optional stopping theorem for the bounded stopping time  $\tau_N=\tau_c \wedge \inf\{t\geq 0:\; x(t) \geq N\}\wedge N$, and letting $N\rightarrow \infty$, we find
$$x_1^2\EE \left[ e^{-\rho \tau_d} \right] -x^2 \leq \EE \Bigg[ \int \limits_0^{\tau_d} e^{-\rho t} \big( -\rho  x^2(t)+2x(t)(d-bx(t)+r(x(t)))+\sigma^2 x^2(t) \big)dt \Bigg]$$   
\noindent
If $m=m(x_1)=-\sup \limits_{x\geq x_1}\{ -bx+r(x) \}$, then $m>0.$ \\Choosing $d\leq m/2$, we have
\begin{equation}\label{x2b}
   x_1^2 -x^2 \leq -(\rho -\sigma_0^2) \EE \left[ \int \limits_0^{\tau_d} e^{-\rho t} x^2(t) dt\right] +\rho x_1^2 \EE \left[ \int \limits_0^{\tau_d} e^{-\rho t} dt\right]
\end{equation}
Combining \eqref{bVx} and \eqref{x2b}, we find
\begin{equation}
    V_{\sigma}(x)-V_{\sigma}(x_1)\geq \left( \ln d -\rho D -\frac{ c\rho x_1^2}{\rho -\sigma_0^2}  \right)  \EE \left[ \int \limits_0^{\tau_d} e^{-\rho t} dt\right] +\frac{c}{\rho-\sigma_0^2} (x_1^2-x^2)
\end{equation}
In order to control the term $\displaystyle{\EE \left[ \int \limits_0^{\tau_d} e^{-\rho t} dt\right]}$, we now apply It$\hat{o}$'s rule to the semimartingale $Z_t=e^{-\rho t}x(t)$, following the same steps as above, to find:
\begin{align*}
x_1\EE \left[e^{-\rho \tau_d}  \right]-x&\leq \EE \left[ \int \limits_0^{\tau_d} e^{-\rho t} \left( -\rho x(t) + d-bx(t) +r(x(t))\right)dt \right]\\
&\leq x_1\EE \left[e^{-\rho \tau_d}  \right]-x_1 -\frac{m}{2}  \EE \left[ \int \limits_0^{\tau_d} e^{-\rho t} dt\right]
\end{align*}
Thus
\begin{equation}\label{btau}
    x_1-x\leq -\frac{m}{2}  \EE \left[ \int \limits_0^{\tau_d} e^{-\rho t} dt\right]
\end{equation}
If we finally choose $d=d(x_1)=\min \left[ \exp \left( \rho D + \frac{c\rho x_1^2}{\rho-\sigma_0^2}\right) , m(x_1) \right]/2$, we find
\begin{equation}
    V_{\sigma}(x)-V_{\sigma}(x_1)\geq \frac{2}{m(x_1)}\left( \ln d(x_1) -\rho D -\frac{c\rho x_1^2}{\rho -\sigma_0^2}  \right) (x-x_1) - \frac{c}{\rho -\sigma_0^2}2x (x-x_1) 
\end{equation}
Dividing by $x-x_1$ and taking the limit $x\rightarrow x_1$, we conclude that
$$
V'_{\sigma}(x) \geq \frac{2}{m(x)}\left( \ln d(x) -\rho D -\frac{c\rho x^2}{\rho -\sigma_0^2}  \right)- \frac{2cx}{\rho -\sigma_0^2} =: \Phi(x)
$$
Therefore, $\tilde{V}'_{\sigma}(x)\geq \Phi(e^x)e^{x}$ for all $x\in \mathbb{R}.$ 

\end{proof}

\begin{lemma} $F_{\sigma}(x)\geq bx+C$ for all $x>1$, where $C$ constant independent of $\sigma.$
\end{lemma}
\begin{proof}
		
		
		

	Let $x>1$.	Since $V'_{\sigma}(x)\leq -C<0$, where $C$ constant independent of $\sigma$ and $r(x)\leq a$ for all $x>0,$
		$$
		F_{\sigma}(x)=\int \limits_{1}^{x} \left( \frac{1}{ V'_{\sigma}(y)}-r(e^y) \right)e^{-y} dy+(b+\frac{\sigma^2}{2})x  > bx+Const
		$$

\end{proof}

\section{Further estimates}\label{further}

We begin by proving a result on the asymptotic behaviour of the derivative of the value function at $+\infty.$  We should highlight that this result is true for the value function of  both the stochastic and the deterministic shallow lake problem. 

\begin{propos}\label{linear}
For  $0\leq \sigma^2 <\rho +2b,$ there exist constants $ M,B>0$, $C\in \mathbb{R}$ such that $$V'(x)\geq -Bx+C \; \; \forall x>M.$$
\end{propos}
\begin{proof}
Let $0<x_1<x_2$ and for $d>0$ consider a control $u\in \mathfrak{U}$ which equals $d$ up to time $\tau_d=\inf \{t\geq 0: \; x(t)\leq x_1 \}$ where $x(\cdot)$ is the solution of \eqref{sldyn} with control $u$ and $x(0)=x_2.$ Then
$$
V(x_2)\geq \EE \left[ \int_{0}^{\tau_d} e^{-\rho t}(\ln u(t) -cx^2(t))dt \right]+\EE \left[ e^{-\rho \tau_d} \right] V(x_1)
$$
\begin{equation}\label{bVxx}	(V(x_2)-V(x_1)) \EE [e^{-\rho \tau_d}] \geq (\ln d-\rho V(x_2)) \EE \left[ \int_{0}^{\tau_d} e^{-\rho t}dt \right]- c\EE \left[ \int_{0}^{\tau_d} e^{-\rho t}x^2(t)dt \right]\end{equation}
Applying It$\hat{o}$'s rule to the semimartingale $Y_t=e^{-\rho t}x^2(t)$, we find
$$
Y_t-x_2^2=\int_{0}^{t} e^{-\rho s} \left[ (\sigma^2-2b-\rho)x^2(s) +2x(s)\left(u(s)+r(x(s))\right) \right]ds+M_1(t)	$$
where $M_1(t)=2\sigma\int_{0}^{t}e^{-\rho s}x^2(s)dW_s$

Applying now the optional stopping theorem for the bounded stopping time  $\tau_N=\tau_d \wedge \inf\{t\geq 0:\; x(t) \geq N\}\wedge N$, we have
$$
\EE [Y_{\tau_N}]-x_2^2= \EE \left[\int_{0}^{\tau_N} e^{-\rho s} \left( (\sigma^2-2b-\rho)x^2(s) +2x(s)\left(d+r(x(s))\right) \right)ds \right]	$$
Since $x(s) \geq x_1$ on $[0,\tau_N]$, we have
$$
x_1^2\EE [e^{-\rho \tau_N}]-x_2^2\leq -\frac{c}{A} \EE \left[\int_{0}^{\tau_N} e^{-\rho s} x^2(s)ds\right] +2(d+a) \EE \left[\int_{0}^{\tau_N} e^{-\rho s} x(s) ds \right]
$$
Letting $N\rightarrow \infty$, we get
$$
x_1^2\EE [e^{-\rho \tau_d}]-x_2^2\leq -\frac{c}{A} \EE \left[\int_{0}^{\tau_d} e^{-\rho s} x_s^2ds\right] +2(d+a) \EE \left[\int_{0}^{\tau_d} e^{-\rho s} x_s ds \right]
$$
\begin{multline*}
(x_1^2-x_2^2)\EE [e^{-\rho \tau_d}]\leq -\frac{c}{A} \EE \left[\int_{0}^{\tau_d} e^{-\rho s} x^2(s)ds\right] +2(d+a) \EE \left[\int_{0}^{\tau_d} e^{-\rho s} x(s) ds \right]\\+\rho x_2^2\EE \left[ \int_{0}^{\tau_d} e^{-\rho s}ds\right]
\end{multline*}
\begin{multline}\label{b1x2}
	-c\EE \left[\int_{0}^{\tau_d} e^{-\rho s} x^2(s)ds\right] \geq A(x_1^2-x_2^2) 	\EE [e^{-\rho \tau_d}]\\-2A(d+a) \EE \left[\int_{0}^{\tau_d} e^{-\rho s} x(s) ds \right]-A\rho x_2^2\EE \left[ \int_{0}^{\tau_d} e^{-\rho s}ds\right]
\end{multline}
Now \eqref{bVxx} gives
\begin{multline}\label{b2Vx}
		(V(x_2)-V(x_1)) \EE [e^{-\rho \tau_d}] \geq  (\ln d-\rho V(x_2)-A\rho x_2^2) \EE \left[ \int_{0}^{\tau_d} e^{-\rho s}ds \right]\\+A(x_1^2-x_2^2)\EE [e^{-\rho \tau_d}]-2A(d+a) \EE \left[ \int_{0}^{\tau_d} e^{-\rho s}x(s)ds \right]
\end{multline}
In order to control the last term in relation \eqref{b2Vx}, we apply It\^o's rule to the semimartingale $\tilde{Y}_t=e^{-\rho t}x(t)$.
$$
\tilde{Y}_t-x_2=\int_{0}^{t} e^{-\rho s} \left(u(s)- (b+\rho)x(s)+r(x(s)) \right)ds +M_2(t)	$$
where $M_2(t)=\sigma\int_{0}^{t}e^{-\rho s}x(s)dW_s$.
Applying again the optional stopping theorem for the bounded stopping time $\tau_N$, we have
$$
\EE [\tilde{Y}_{\tau_N}]-x_2=\EE \left[ \int_{0}^{\tau_N} e^{-\rho s} \left(d- (b+\rho)x(s)+r(x(s)) \right)ds \right]
$$
Since $x(s) \geq x_1$ on $[0,\tau_N]$, letting $N\rightarrow \infty$, we have
$$
x_1\EE [e^{-\rho\tau_d}]-x_2\leq \EE \left[ \int_{0}^{\tau_d} e^{-\rho s} \left(d- (b+\rho)x(s)+r(x(s)) \right)ds \right]
$$
\begin{multline}\label{drvbound}
	(x_1-x_2)\EE [e^{-\rho\tau_d}]\leq (d+\rho x_2) \EE \left[ \int_{0}^{\tau_d} e^{-\rho s}ds\right]\\ + \EE \left[ \int_{0}^{\tau_d} e^{-\rho s} \left(- (b+\rho)x(s)+r(x(s)) \right)ds \right]
\end{multline}

\begin{equation}\label{bx}
	(x_1-x_2)\EE [e^{-\rho\tau_d}]\leq (d+\rho x_2+a) \EE \left[ \int_{0}^{\tau_d} e^{-\rho s}ds\right] -(b+\rho) \EE \left[ \int_{0}^{\tau_d} e^{-\rho s} x(s)ds \right]	
\end{equation}

\begin{equation}\label{b2x}
	-\EE \left[ \int_{0}^{\tau_d} e^{-\rho s} x(s)ds \right] \geq \frac{1}{b+\rho}	(x_1-x_2)\EE [e^{-\rho\tau_d}]\\ - \frac{d+\rho x_2+a}{b+\rho} \EE \left[ \int_{0}^{\tau_d} e^{-\rho s}ds\right]	
\end{equation}

Relation \eqref{b2Vx} based on \eqref{b2x} becomes:
\begin{multline}\label{b3Vx}
		(V(x_2)-V(x_1)) \EE [e^{-\rho \tau_d}]  \geq  \\ \left(\ln d-\rho V(x_2)-A\rho x_2^2- \frac{2A(d+a)(d+\rho x_2+a)}{b+\rho} \right)  \EE \left[ \int_{0}^{\tau_d} e^{-\rho t}dt \right]\\+A(x_1^2-x_2^2)\EE [e^{-\rho \tau_d}] 
		+\frac{2A(d+a)}{b+\rho}(x_1-x_2) \EE \left[ e^{-\rho \tau_d} \right]
\end{multline}
If we denote by $g(d;x_2)$ the coefficient of $ \displaystyle{\EE \left[ \int_{0}^{\tau_d} e^{-\rho t}dt \right]}$, we can consider the following two possible scenarios.
\begin{itemize}
	\item Case 1: If $\max \limits_{d>0}g(d;x_2)\geq 0$, we choose $d=d(x_2)=\argmax \limits_{d>0}g(d;x_2)$ and relation \eqref{b3Vx} gives:
	\begin{align*}
	\frac{V(x_2)-V(x_1)}{x_2-x_1}&\geq -A(x_1+x_2)-\frac{2A(d(x_2)+a)}{b+\rho} \\& \geq -2Ax_2 -\frac{2A(d(x_2)+a)}{b+\rho}
	\end{align*}
	Notice that $d(x_2)$ is bounded from above, since $d(x_2)$ is such that $g'(d;x_2)=0$, which gives that
	\begin{equation*}
	d(x_2)=\frac{-2A(\rho x_2+2a)+\sqrt{4A^2(\rho x_2+2a)^2+16A(b+\rho)}}{8A} \leq \sqrt{\frac{b+\rho}{4A}}=:d_1	    
	\end{equation*}
	Thus, 
	\begin{equation}\label{f1vx}
		\frac{V(x_2)-V(x_1)}{x_2-x_1}\geq -2Ax_2-\frac{2A(d_1+a)}{b+\rho}	
	\end{equation}
	\item Case 2: If $g(d;x_2)<0 \; \; \forall d>0$, then focusing on relation \eqref{bx} and choosing $\frac{2(\rho+a)}{b}<x_1<x_2<x_1+1$, we have that for all $x\geq x_1$:
	$$
	d+\rho x_2+a-(b+\rho)x\leq d+\rho(x_1+1)+a-(b+\rho)x_1<d-(\rho+a)
	$$
	So if we choose $d=\rho,$ relation \eqref{bx} becomes:
	\begin{equation}\label{b3x}
		(x_1-x_2)\EE [e^{-\rho\tau_d}]\leq - a\EE \left[ \int_{0}^{\tau_d} e^{-\rho s}ds \right]
	\end{equation}
	\begin{equation}\label{bvxsf}
		\eqref{b3Vx} \overset{\eqref{b3x}}{\Rightarrow} \frac{V(x_2)-V(x_1)}{x_2-x_1}\geq \frac{g(\rho ;x_2)}{a}-\frac{2A(\rho+a)}{b+\rho}-2Ax_2
	\end{equation}
	Since $V(x)+Ax^2$ is decreasing by Proposition 2.2(ii) \cite{KL}, we have:
	\begin{equation}\label{bdg}
		g(\rho;x_2)\geq \ln \rho -\rho V(0)-\frac{2A(\rho+a)^2}{b+\rho}-\frac{2A\rho(\rho+a)}{b+\rho}x_2
	\end{equation}
 Relation \eqref{bvxsf} based on \eqref{bdg} becomes
	\begin{equation}\label{bvxf}
		\frac{V(x_2)-V(x_1)}{x_2-x_1}\geq \frac{1}{a}(\ln \rho -\rho V(0))-\frac{2A(\rho+a)(\rho+2a)}{a(b+\rho)}-\left(\frac{2A\rho(\rho+a)}{a(b+\rho)}+2A\right)x_2
	\end{equation}
\end{itemize}
The assertion now follows by taking $B=\frac{2A\rho(\rho+a)}{a(b+\rho)}+2A$, $M= \frac{2(\rho \vee d_1+a)}{b}$ and $C=\min\{\frac{1}{a}(\ln \rho -\rho V(0))-\frac{2A(\rho+a)(\rho+2a)}{a(b+\rho)},-\frac{2A(d_1+a)}{b+\rho}\}$.
\end{proof}

The fact that $V'$ does not go to minus infinity more quickly than linearly along with its upper bound is enough to establish the boundedness of the second derivative of $V$ at $+\infty$ as it is stated in the following corollary.

\begin{clr}\label{secderbound}
    For the value function, $V$, of the stochastic shallow lake problem, we have
    $$
   -\infty<\liminf \limits_{x\rightarrow \infty} V''(x)\leq \limsup \limits_{x\rightarrow \infty} V''(x)<\infty
    $$
\end{clr}
\begin{proof}
    Since $V$ is a classical solution to \eqref{OHJB}, we have that
    \begin{equation}\label{secder}
    V''(x)=\frac{2}{\sigma^2} \left( \rho \frac{V(x)}{x^2}- (\frac{r(x)}{x}-b)\frac{V'(x)}{x}+\frac{\ln (-V'(x))}{x^2}+c+\frac{1}{x^2} \right)     
    \end{equation}
    Based on Proposition 2.3(i) \cite{KL}, $\lim \limits_{x\rightarrow \infty}\frac{V(x)}{x^2}=-A$. Furthermore, based on Proposition \ref{linear} and 2.3(ii) \cite{KL}, for $x\gg 1$, $-Bx+C\leq V'(x)\leq -C_1$. Therefore, relation \eqref{secder} gives
    $$
    \frac{2}{\sigma^2}\left(-\rho A-bB+c\right)\leq \liminf \limits_{x\rightarrow \infty}V''(x)\leq \limsup \limits_{x\rightarrow \infty}V''(x)\leq \frac{2}{\sigma^2}\left(-\rho A+c\right)
    $$
\end{proof}

We now proceed by proving properties of the noiseless value function. Proposition \ref{regsecder} refers to the regularity of $V$.

\begin{propos}\label{regsecder}
The value function, $V$, of the deterministic shallow lake problem is  $C^2((0,\infty)\setminus \{x_*\})$ where $x_*$ is the Skiba point.
\end{propos}
\begin{proof}
According to the analysis of subsection \ref{det_ex}, it follows that the value function $V$ is identical to the function $J_P$ constructed based on the Pontryagin Maximum Principle. Along the optimal solution of the system of the lake \eqref{slakes}, we have that $f(x,u)\neq 0$ for all $x\neq x_-,x_+, x_*$. Therefore, based on the Implicit Function Theorem and the HJB eq. \eqref{OHJB}, it follows that the value function, $V$, is $C^2\left([0,\infty)\setminus \{x_-,x_+, x_*\}\right).$ Since along the optimal trajectories $\frac{dV}{dx}=-\frac{1}{u}$ and $\frac{du}{dx}=\frac{g(u,x)}{f(u,x)}$, it remains to prove that there exists the limit $\lim \limits_{(x,u)\rightarrow (x_0,u_0)}\frac{g(u,x)}{f(u,x)}$ and it is finite, where for simplicity reasons we denote by $(x_0,u_0)$ the saddle steady states of the system \eqref{slakes}.
	The point $(x_0,u_0)$ satisfies:
	\begin{equation}\label{saddle}
		\begin{cases}
			u_0=bx_0-r(x_0)\\
			u_0=\frac{b+\rho}{2cx_0}-\frac{r'(x_0)}{2cx_0}
		\end{cases}
	\end{equation}
	
	The linear approximation around $(x_0,u_0)$ gives:	
		
	\begin{equation}\label{linearap}
		d\begin{pmatrix}
			x-x_0\\
			u-u_0
			\end{pmatrix}= \underbrace{ \begin{pmatrix}
			-b+r'(x_0) & 1\\
			2u^2_0+\frac{2\left(1-3x^2_0\right)}{\left(x^2_0+1\right)^3}u_0 & -(b+\rho)+\frac{2x_0}{\left(x^2_0+1\right)^2}+4x_0u_0
		\end{pmatrix} }_{A}
	\begin{pmatrix}
		x-x_0\\
		u-u_0
	\end{pmatrix} 
	\end{equation}	
		In order to determine the direction of the optimal trajectory close to $(x_0,u_0)$, we need to compute the eigenvector which corresponds to the negative eigenvalue of $A$.
If we denote by $\la^*<0$, the negative eigenvalue of $A$, then the corresponding eigenvector is
\begin{equation}
	v^*=\begin{pmatrix}
		1\\
		b-r'(x_0)+\la^*
	\end{pmatrix}
\end{equation}
Therefore, close to $(x_0,u_0)$ along the optimal trajectory we have that
\begin{equation}\label{dirc}
u-u_0=	\underbrace{\left(b-r'(x_0)+\la^*\right)}_{k(u_0,x_0)}(x-x_0)
\end{equation}
By a Taylor expansion along the optimal trajectory we have
\[
g(u,x)=\big[\partial_u g (u_0,x_0)k(u_0,x_0)+\partial_x g (u_0,x_0) \big] (x-x_0)+o(x-x_0), \quad \text{as } x\to x_0
\]
and
\begin{align*}
f(u,x)&=\big[\partial_u f (u_0,x_0)k(u_0,x_0)+\partial_x f (u_0,x_0) \big] (x-x_0)+o(x-x_0)\\
 &=\big[k(u_0,x_0)-b+r'(x_0)\big] (x-x_0)+ o(x-x_0)=\lambda^*(x-x_0)+o(x-x_0), \text{ } x\to x_0.
\end{align*}
Hence,
\begin{align*}
\lim \limits_{(x,u)\rightarrow (x_0,u_0)}\frac{g(u,x)}{f(u,x)} &=\frac{\partial_u g (u_0,x_0)k(u_0,x_0)+\partial_x g (u_0,x_0)}{\la^*}\in\mathbb{R}.
\end{align*}

%
\end{proof}

Now, Lemma \ref{final} collects all the key properties of $V$ to establish the boundedness of its second derivative, whereat the second derivative exists.  

\begin{lemma}\label{final}
For the value function, $V$, of the deterministic shallow lake problem we have that 
\begin{enumerate}[(i)]
    \item $\lim \limits_{x\rightarrow \infty} \frac{V'(x)}{x}= -{2}{A}$
    \item $\lim \limits_{x\rightarrow \infty}V''(x) = -2A$
    \item $V''(0)=-(\rho+b-r'(0))\left(V'(0)\right)^2$
    \item the limits $\lim \limits_{x\rightarrow x_*-} V''(x)$ and $\lim \limits_{x\rightarrow x_*+} V''(x)$ exist and are finite.
\end{enumerate}
\end{lemma}
\begin{proof}
    \begin{enumerate}[(i)]
        \item Since $V$ is a classical solution to the HJB eq. \eqref{OHJB} for $x>x_*$, we have
        \[
        \begin{split}
        \lim \limits_{x\rightarrow \infty} \frac{V'(x)}{x}&=\lim \limits_{x\rightarrow \infty} \frac{x}{bx-r(x)}\left(-\frac{\rho V(x)}{x^2}-\frac{\ln (-V'(x))}{x^2}-c-\frac{1}{x^2} \right)\\&=\frac{1}{b}(A\rho-c)=-2A,
        \end{split}
        \]
        where in the last line we used  Proposition 2.3 \cite{KL} and  \ref{linear}.
         \item By relation \eqref{secderzero}, we find
        \[
        \begin{split}
        \lim \limits_{x\rightarrow \infty}V''(x)&=
        \lim \limits_{x\rightarrow \infty} \frac{\left(r'(x)-(\rho+b)\right) \frac{ V'(x)}{x}-2c}{-\frac{r(x)}{x}+b+ \frac{1}{xV'(x)}}\\
        &= \frac{2 A(\rho+b)-2c}{b}=-2A.
        \end{split}
        \]
        \item Differentiating once the HJB eq. \eqref{OHJB} for $x$ close to $0$, we find
        \begin{equation}\label{secderzero}
	V''(x)=\frac{\left(r'(x)-(\rho+b)\right) V'(x)-2cx}{-r(x)+bx+ \frac{1}{V'(x)}}  .         
        \end{equation}
        Taking $x\rightarrow 0$, the assertion follows since $V$ is $C^1([0,\infty)\setminus {x_*})$ and $C^2((0,\infty)\setminus {x_*})$. 
        \item  Because of \eqref{secderzero} and the fact that the side derivatives $V'_P(x_*(-))$, $ V'_P(x_*(+))$ exist (see Proposition 3.8 \cite{KosZoh}), we just need to show that the limit of the denominator of \eqref{secderzero}  as $x\rightarrow x_*(\pm)$ is not zero. This is an immediate consequence of Lemma \ref{genres}.
    \end{enumerate}
\end{proof}



\noindent
\textbf{Acknowledgements}\\
The authors wish to thank  Iasson Karafyllis for useful discussions.  \\

\noindent
The implementation of this work was co-financed by Greece and the
European Union (European Social Fund-ESF) through the Operational Programme ``Human Resources Development, Education and Lifelong Learning'' in the context of the Act ``Enhancing Human Resources Research Potential by undertaking a Doctoral Research'' Sub-action 2: IKY Scholarship Programme for PhD candidates in the Greek Universities.

\bibliographystyle{plain} 
\bibliography{reference}


\end{document}